\documentclass[twoside,12pt]{article}

\usepackage{amsmath}
\usepackage{amssymb}
\usepackage{amscd}
\usepackage{amsthm}

\numberwithin{equation}{section}

\theoremstyle{plain}
\newtheorem{theorem}{Theorem}[section]

\newtheorem{Prop}[theorem]{Proposition}
\newtheorem{Le}[theorem]{Lemma}
\newtheorem{Cor}[theorem]{Corollary}

\theoremstyle{remark}

\theoremstyle{definition}

\newcommand{\R}{\mathbb R}

\def\ga{\gamma}

\def\ra{\rightarrow}

\def\e{\emph}
\def\i{\infty}
\def\p{\partial}
\def\b{\begin}

\begin{document}

\title{
{Quasisymmetric  Maps on the Boundary of a Negatively Curved
Solvable Lie Group}}
\author{Xiangdong Xie}
\date{  }

\maketitle

\begin{abstract}
We describe all the self quasisymmetric maps on the  ideal
    boundary
 of a particular  negatively curved solvable Lie group. As applications, we
 prove a Liouville type theorem,  and
 derive some rigidity properties
   for quasiisometries of the solvable Lie group.

\end{abstract}




 {\small {\bf{Mathematics Subject
Classification (2000).}} 20F65,  30C65, 53C20.









\setcounter{section}{0} \setcounter{subsection}{0}

\section{Introduction}\label{s0}

In this paper we study quasisymmetric maps on the ideal boundary of
a particular  negatively curved solvable Lie group.

Let
\[A=\left(\begin{array}{cc} 1 & 1\\ 0 & 1\end{array}\right).
\]
 Let $\R$ act on $\R^2$ by
             $(t, v)\rightarrow e^{tA} v$ ($t\in \R$, $v\in \R^2$).
             We denote the corresponding semi-direct product by
             $G_A=\R^2\rtimes_A \R$. That is,  $G_A=\R^2\times\R$
             as a   smooth manifold, and the group operation is
             given by:

             $$(v,t)\cdot (w, s)=(v+e^{tA}w, t+s)$$
               for all $(v,t), (w, s)\in\R^2\times\R$.
               The group $G_A$  is a
simply connected  solvable Lie group.  

We endow $G_A$ with the left invariant
   Riemannian
 metric determined by taking
the standard Euclidean metric at the identity of
$G_A=\R^2\times\R=\R^{3}$.
With this  metric  $G_A$ has  pinched negative sectional curvature
(and so is Gromov hyperbolic). Hence $G_A$ has a well defined ideal
boundary $\p G_A$. There is a so-called cone topology on
$\overline{G_A}=G_A\cup \p G_A$, in which $\p G_A$ is homeomorphic
to the $2$-dimensional sphere and $\overline{G_A}$ is homeomorphic
to the closed $3$-ball in the Euclidean space. For each $v\in \R^2$,
the map $\gamma_v: \R\ra G_A$, $\gamma_v(t)=(v,t)$  is a geodesic.
We call such a geodesic a vertical geodesic.  It can be checked that
all vertical geodesics are asymptotic as $t\ra +\infty$. Hence they
define a point $\xi_0$ in the ideal boundary $\p G_A$.

Each geodesic ray in $G_A$ is  asymptotic to either  an upward
oriented vertical geodesic or a downward oriented vertical geodesic.
The upward oriented geodesics are asymptotic to $\xi_0$ and the
downward oriented vertical  geodesics are in 1-to-1 correspondence
with $\R^2$. Hence $\p G_A\backslash\{\xi_0\}$ can be naturally
identified with $\R^2$.

For any proper Gromov hyperbolic geodesic space $X$ and any $\xi\in
\p X$, there are so-called parabolic visual (quasi)metrics
 on $\p X\backslash \{\xi\}$. See \cite{SX}, Section 5.  In our case, a   parabolic
 visual quasimetric   $D$ on $\p G_A\backslash \{\xi_0\}$  is given by:
$$D((x_1, y_1), (x_2, y_2))=\max\big\{\big\vert y_2-y_1\big\vert, \big\vert(x_2-x_1)-(y_2-y_1)\ln
     |y_2-y_1|\big\vert\big\}$$
       for all $(x_1,  y_1), (x_2, y_2)\in \R^2=\p G_A\backslash
       \{\xi_0\}$, where
                 $0\ln 0$ is understood  to be $0$.

We remark that $D$ is not a metric on $\R^2$, but merely a
quasimetric. Recall that  a quasimetric $\rho$  on a set $A$
  is a function $\rho:  A\times A\ra \R$ satisfying the following
  three conditions:   (1)  $\rho(x,y)=\rho(y,x)$ for all $x, y\in
  A$;
    (2)  $\rho(x, y)\ge 0$  for all $x,y\in A$  and
    $\rho(x, y)=0$   if  and only  if  $x=y$;
       (3)   there is some $M\ge 1$ such that $\rho(x, z)\le M (\rho(x,
y)+\rho(y,z))$ for all $x, y, z\in A$. For each   $M\ge 1$,
there is a constant $\epsilon_0>0$ such that $\rho^\epsilon$ is
biLipschitz
   equivalent to
    a metric
for all    quasimetric $\rho$  with constant $M$  and   all
 $0<\epsilon\le \epsilon_0$,
     see Proposition 14.5. in \cite{Hn}.

Let $\eta: [0,\i)\ra [0,\i)$ be a homeomorphism.   A bijection
$F:X\to Y$ between two quasimetric spaces is
\e{$\eta$-quasisymmetric} if for all distinct triples $x,y,z\in X$,
we have
\[
   \frac{d(F(x), F(y))}{d(F(x), F(z))}\le \eta\left(\frac{d(x,y)}{d(x,z)}\right).
\]
    A map $F:  X\to Y$ is quasisymmetric if it is $\eta$-quasisymmetric
for some $\eta$.


The following is the main result of the paper.

\b{theorem}\label{intromain}
 {Every quasisymmetric map $F: (\R^2, D)\ra (\R^2, D)$  is a
 biLipschitz map.
   Furthermore, a bijection $F:(\R^2, D)\ra (\R^2, D)$
   is a quasisymmetric map   if  and only of it
has the following form: $F(x,y)=(ax+c(y), ay+b)$
  for all $(x,y)\in \R^2$, where $a\not=0$,  $b$ are constants and
  $c: \R\ra \R$ is a Lipschitz map.}
\end{theorem}

One should compare this with quasiconformal maps on the sphere or
the Euclidean space, where there are plenty of non-biLipschitz
quasiconformal  maps. On the other hand, the conclusion of
Theorem~\ref{intromain} is not as strong as in the cases of
quarternionic
  hyperbolic spaces,
Cayley plane (\cite{P2}) and Fuchsian buildings (\cite{BP},
\cite{X}), where every quasisymmetric map of the ideal boundary is
actually a conformal map. In our case, there are many  non-conformal
quasisymmetric maps of the ideal boundary of $G_A$.

As applications, we describe all the isometries and all the
similarities of $(\R^2, D)$, see Proposition \ref{iso.and}.
     We also prove a
Liouville type theorem for $(\R^2, D)$.

\b{theorem}\label{introliouville}
  {Every conformal map  $f: (\R^2, D)\ra (\R^2, D)$ is the boundary
  map of an isometry $G_A\ra G_A$.}
  \end{theorem}

  Theorem \ref{intromain} also has  geometric consequences.
Let $L\ge 1$ and $C\ge 0$. A (not necessarily continuous ) map
$f:X\ra Y$ between two metric spaces is an $(L,A)$-\e{quasiisometry}
if:
\newline (1) $d(x_1,x_2)/L-C\le d(f(x_1), f(x_2))\le L\, d(x_1, x_2)+C$
for all $x_1,x_2\in X$;
\newline (2)  for any $y\in Y$, there is some $x\in X$ with
$d(f(x), y)\le C$.
\newline
In the case   $L=1$, we call  $f$  an \e{almost isometry}.

  \begin{Cor}\label{c0}
Every self quasiisometry of $G_A$ is an almost isometry.
\end{Cor}

Notice that an almost isometry is not necessarily a finite distance
away from an isometry.

\noindent {\bf{Acknowledgment}}. {I would like to thank Bruce
Kleiner for  stimulating  discussions.  I thank him for  allowing
    me to include his proof of the fact that quasisymmetric maps
preserve the horizontal foliation (Section 3).
 I also
would like to thank the Department of Mathematical Sciences at
Georgia Southern University for generous travel support.}

\section{Quasimetrics on the ideal boundary}\label{metriconb}

In this  section,  we  will define three different  parabolic visual
  quasimetrics on the
 ideal boundary, and
 find an explicit  formula for
one of them.  The three quasimetrics are biLipschitz equivalent with
each other.

Let $A$ and $G_A$
  be as in the Introduction.
We endow $G_A$ with the left invariant metric determined by taking
the standard Euclidean metric at the identity of
$G_A\approx\R^2\times\R=\R^3$.
At a point $(x,t)\in\R^2\times\R\approx G_A$, the tangent space is
identified with $\R^2\times\R$, and  the Riemannian metric   is
given by the symmetric matrix
  \[\left(\begin{array}{cc} Q_A(t) & 0\\ 0 & 1\end{array}
    \right),
  \]
where $Q_A(t)=e^{-tA^T}e^{-tA}$.  Here $A^T$ denotes the transpose
of $A$.  With this metric $G_A$ has sectional curvature
$-(6+\sqrt{29})/4=-b^2\le K\le -a^2=-(6-\sqrt{29})/4$. Hence $G_A$
has a well defined ideal boundary $\p G_A$. All vertical geodesics
$\gamma_v$ ($v\in \R^2$) are asymptotic as $t\ra +\infty$. Hence
they define a point $\xi_0$ in the ideal boundary $\p G_A$.

    The sets $\R^2\times\{t\}$  ($t\in \R$)
are horospheres centered at $\xi_0$.
      For each  $t\in \R$, the
induced metric on  the horosphere $ \R^2\times\{t\}\subset G_A$   is
determined by the quadratic form $Q_A(t)$. This metric has distance
formula $d_{\R^2\times \{t\}}((v,t), (w,t))=| e^{-tA}(v-w)|$.  Here
$|\cdot |$ denotes the Euclidean norm.

Each geodesic ray in $G_A$ is  asymptotic to either  an upward
oriented vertical geodesic or a downward oriented vertical geodesic.
The upward oriented geodesics are asymptotic to $\xi_0$ and the
downward oriented vertical  geodesics are in 1-to-1 correspondence
with $\R^2$. Hence $\p G_A\backslash\{\xi_0\}$ can be naturally
identified with $\R^2$.

We  next define three  parabolic visual quasimetrics on
  $\p G_A\backslash\{\xi_0\}=\R^2$.
 Given $v,
w\in\R^2\approx \p G_A\backslash\{\xi_0\}$, the parabolic visual
quasimetric $D_e(v,w)$ is defined as follows: ${D}_e(v,w)=e^t$,
where $t$ is the unique real number such that at height $t$ the two
vertical geodesics $\ga_v$ and $\ga_w$ are at distance one apart in
the horosphere; that is, $d_{\R^n\times \{t\}}((v,t), (w,t))=|
e^{-tA}(v-w)|=1.$ Here the subscript {\it{e}} in $D_e$  means it
corresponds to the Euclidean norm.

  Recall that the super norm
   on $\R^2$
  is given by:
 $|(x,y)|_s=\max\{|x|, |y|\}$  for all $(x,y)\in \R^2$.
   The  parabolic visual quasimetric  $D_s$ on $\p G_A \backslash
\{\xi_0\}$ is defined as follows: ${D}_s(v,w)=e^t$, where $t$ is the
smallest  real number such that at height $t$ the two vertical
geodesics $\ga_v$ and $\ga_w$ are at distance one apart with respect
to the norm $|\cdot|_s$; that is, $| e^{-tA}(v-w)|_s=1.$   Here the
subscript {\it{s}} in  $D_s$  means it corresponds to the super
  norm $|\cdot|_s$.

Notice that $|v|_s\le |v|\le \sqrt{2}   \, |v|_s$  for all $v\in
\R^2$.
 Using this, one  can  verify the following  lemma,
 whose proof is left to the reader.

 \b{Le}\label{norms}
 {For all $v, w\in \R^2$ we have
 $D_s(v,w)\le D_e(v,w)\le 2^{1/{2a}} D_s(v,w)$,  where
   $a=\sqrt{6-\sqrt{29}}/2$.}
\end{Le}

 The following result provides a parabolic visual quasimetric   $D$
     which
 admits an explicit formula and is also biLipschitz equivalent with
  $D_e$ and $D_s$.

\begin{Prop}\label{metric}
For all $v=(x_1, y_1),  w=(x_2,  y_2)\in \R^2$,
   $$D(v, w)/3\le  D_s(v,w)\le 3 D(v, w),$$
     where $D(v, w)=\max\big\{\big\vert y_2-y_1\big\vert, \big\vert(x_2-x_1)-(y_2-y_1)\ln
     |y_2-y_1|\big\vert\big\}$  and      $0\ln 0$ is understood  to be $0$.

\end{Prop}

Let $g=((x,y), t)\in \R^2\times \R$ and denote by $L_g:
G_A\rightarrow G_A$  the left translation by $g$.  We calculate
\[L_g((x',y'), t')=((x+e^t(x'+ty'), y+e^ty'), t'+t).\]
  We  see that  $L_g$ maps  vertical geodesics   to  vertical  geodesics.
   It follows that $L_g$ induces a map   $T_g: \R^2\rightarrow \R^2$,
     $$T_g(x', y')=(x+e^t(x'+ty'), y+e^ty').$$
      Since $L_g$ is an isometry of $G_A$ and it translates by $t$
      in the vertical direction, the definition of
        the quasimetric $ D_e$ shows that
        $$ D_e(T_g(x_1,y_1), T_g(x_2,y_2))=e^t
         D_e((x_1,y_1),(x_2,y_2))$$
          for all $(x_1,y_1), (x_2, y_2)\in \R^2$.
         In other words,
         $T_g$ is a similarity of $(\R^2, D_e)$ with similarity constant
         $e^t$.
            When $t=0$, $T_g$ is simply a Euclidean translation  and
            it is an isometry with respect to
             $ D_e$. Similar statements also hold for the
             quasimetric $D_s$.

             Notice that Euclidean translations are also isometries with respect to the function
             $D$.
               This together with the same statement about $ D_s$
               implies that we can assume $(x_1, y_1)=(0,0)$ in
               order to prove
               Proposition \ref{metric}.

\begin{proof}[\bf{Proof of   Proposition~\ref{metric}.}]
 By the preceding remark, we may assume $(x_1, y_1)=(0,0)$, and
write $(x,y)$ for $(x_2, y_2)$.
  Recall $ D_s((x,y), (0,0))=e^t$ if
 $t$ is the smallest real  number such that
  $|e^{-tA}(x,y)|_s=1.$
We calculate  $| e^{-tA}(x,y)|_s=\max\{e^{-t}|x-ty|, e^{-t}|y|\}$.
 We consider several cases:

Case 1: $y=0$. In this case, $| e^{-tA}(x,0)|_s=e^{-t}|x|$. Hence
 $ D_s((x,0),(0,0))=e^t=|x|=D((x,0),(0,0))$.

When $y\not=0$,    we let $t_0=\ln|y|$  and $a=x/y-\ln|y|$.

 Case 2:  $y\not=0$  and  $|x-t_0 y|\le |y|$.
  In this case,
  $|e^{-t_0A}(x,y)|_s=\max\{e^{-t_0}|x-t_0y|, e^{-t_0}|y|\}=e^{-t_0}|y|=1$.
    Notice also $| e^{-tA}(x,y)|_s\ge e^{-t}|y|>1$ if $t<t_0$.
    Hence $ D_s((x,y),(0,0))=e^{t_0}=|y|=D((x,y),(0,0))$.

When  $y\not=0$  and  $|x-t_0 y|> |y|$, we have $|a|>1$.

Case 3:    $y\not=0$,    $|x-t_0 y|> |y|$  and $a>1$. In this case,
   $D((x,y),  (0,0))=|x-y\ln |y||$.
  Let
  $t_1>t_0$ be the smallest real number $t$ satisfying
 $e^{-t}|x-ty|=1$.
  Notice that $e^{-t_1}|y|<1$  and so
   $ D_s((x,y),(0,0))=e^{t_1}$.
  Set $u=t_1-t_0>0$.  The equality  $e^{-t_1}|x-t_1y|=1$ implies
  $e^u=a-u$.   Clearly $e^u=a-u\le a$.
   We claim $e^u=a-u\ge a/3$.  Otherwise,
  $u>2a/3$. This contradicts $a=u+e^u>u+(1+u)$.
    Hence $a/3\le e^u\le a$
      and
      $$|x-y\ln |y| |/3=a |y|/3\le |y| e^u=e^{t_1}= D_s((x,y),(0,0))\le a|y|= |x-y\ln
      |y||.$$

      Case 4:   $y\not=0$,    $|x-t_0 y|> |y|$  and $a<-1$. In this case,
  $D((x,y),  (0,0))=|x-y\ln
      |y||$.
     Let $t_1>t_0$ be the smallest real number $t$ satisfying
 $e^{-t}|x-ty|=1$.
   Again we have $ D_s((x,y),(0,0))=e^{t_1}$.
  Set $u=t_1-t_0>0$.  The equality  $e^{-t_1}|x-t_1y|=1$ implies
  $e^u=u-a$.   Clearly $e^u=u-a>-a$.
We claim $e^u=u-a\le -3a$.
 Otherwise,  $u>-2a$  and hence  $-a=e^u-u>1+u^2/2>u>-2a$, a
 contradiction.
    Hence $|a|=-a\le e^u\le-3a=3|a|$
  and
    $$|x-y\ln |y||= |ay|\le |y| e^u=e^{t_1}=D_s((x,y),(0,0))\le 3|ay|=  3|x-y\ln
      |y||.$$

\end{proof}

We  describe  some isometries and similarities of the space $(\R^2,
D)$. The following proposition can be easily proved by using the
formula for $D$.

\b{Prop}\label{iso.and.simi}
  Let $(\R^2, D)$ be as above.\newline
  (1) Then Euclidean translations of $\R^2$ are isometries with
  respect to $D$;\newline
  (2)   Let $\pi: \R^2 \ra \R^2$ be defined by $\pi(x,y)=(-x, -y)$.
   Then $\pi$ is an isometry  with
  respect to $D$;  \newline
  (3) For any  real   number $t$, let $\lambda_t: \R^2\ra \R^2$ be
  defined by $\lambda_t(x,y)=(e^t(x+ty), e^ty)$.  Then
    $D(\lambda_t(x_1, y_1),  \lambda_t(x_2, y_2))=e^t\cdot D((x_1, y_1),
    (x_2, y_2))$  for all $(x_1, y_1), (x_2, y_2)\in \R^2$.

    \end{Prop}

We notice that the three classes of maps in Proposition
\ref{iso.and.simi}  are boundary maps of isometries of $G_A$.
  Euclidean translations of $\R^2=\p G_A\backslash \{\xi_0\}$
    are boundary maps of left translations   $L_g$  of $G_A$
   for elements $g$ of type $g=((x,y), 0)\in \R^2\times \R=G_A$.
      The map $\lambda_t$ is the boundary map of left translation
       $L_g$ for $g=((0,0), t)$. Finally, $\tau$ is the boundary
       map of the  automorphism   $\tau':G_A\ra G_A$,
       $\tau'((x,y),t)=((-x, -y), t)$.  Notice that $\tau'$ is indeed
       an automorphism of $G_A$ and the tangential map of $\tau'$ at
       the identity is  an isometry. It follows that $\tau'$ is an
       isometry of $G_A$.

\section{Quasisymmetric maps preserve horizontal foliation}\label{hori}

In this section we prove that every self quasisymmetric map  of
$(\R^2, D)$   maps  horizontal    lines to horizontal lines.
 The proof belongs to Bruce Kleiner.  Here I am trying to provide
 more details and I am responsible for the inaccuracies that might
 result from this.
 I would like to
  express my gratitude towards Bruce for allowing me to include
   his argument.

\b{Def}\label{quasiball} {Let $(X, \rho)$ be a quasimetric space and
$L\ge 1$.
 A subset $A\subset X$ is  called an $L$-\e{quasi-ball} if there is
 some $x\in X$ and some $r>0$ such that $B(x,r)\subset A\subset B(x,
 Lr)$.    Here $B(x,r)=\{y\in X: \rho(y,x)<r\} $.}   

 \end{Def}

   The following notion is key to the proof.

   \b{Def} (Kleiner) \label{varation}
   {Pick $Q\ge 1$. Let $u: X\ra \R$ be a function (not necessarily
   continuous) defined on a quasimetric space, and let ${\mathcal{P}}$
   be a  collection of subsets of $X$.  The \e{$Q$-variation of
   $u$  over ${\mathcal{P}}$} --  denoted  $V_Q(u,{\mathcal{P}})$ -- is the quantity
$$\Sigma_{P\in {\mathcal{P}}}[osc(u|_P)]^Q,$$
  where  $osc(u|_p)$  denotes  the oscillation (sup minus inf)  of
  the restriction of $u$  to the  subset  $P\subset X$. The
  \e{$Q$-variation   $V_Q(u)$   of $u$}  is $\sup\{V_Q(u,{\mathcal{P}})\}$ where
  ${\mathcal{P}}$  ranges over all disjoint collections   of balls
  in $X$.  The \e{$(Q,K)$-variation $V_{Q,K}(u)$   of $u$}  is
     $\sup\{V_Q(u, {\mathcal{P}})\}$ where ${\mathcal{P}}$
       ranges over  all disjoint  collections  of $K$-quasi-balls
       in $X$.}
       \end{Def}

There are useful variants of this definition, for instance one can
look at the infimum over all coverings. Or one can take the infimum
over all coverings followed by the sup as the mesh size tends to
zero.   The definition preforms the same function as Pansu's modulus
    \cite{P1}, but it seems  easier to digest.

\b{Le}\label{l3.1}
   {Let $F:  X\ra Y$ be an $\eta$-quasisymmetric map between two
   quasimetric spaces.   Then  for every function  $u:  X\ra \R$ we
   have
    $V_{Q, K}(u)\le V_{Q, \eta(K)}(u\circ F^{-1})$.

   }

   \end{Le}

   \b{proof}
For any subset $A\subset X$, the oscillation of $u$ on $A$ equals
the oscillation of $u\circ F^{-1}$  on $F(A)$.  Let ${\mathcal{P}}$
be a disjoint collection of $K$-quasi-balls in $X$. Then
 $F({\mathcal{P}})=\{F(P):  P\in {\mathcal{P}}\}$ is a disjoint
 collection of $\eta(K)$-quasi-balls in $Y$,  and
$V_Q(u,{\mathcal{P}})=V_Q(u\circ F^{-1},  F({\mathcal{P}}))$.
  Hence
  $$V_{Q,K}(u)=\sup_{\mathcal{P}}V_Q(u, {\mathcal{P}})
  =\sup_{\mathcal{P}}V_Q(u\circ F^{-1},  F({\mathcal{P}}))
  \le V_{Q,\eta(K)}(u\circ F^{-1}).$$

   \end{proof}

  By
   Proposition \ref{iso.and.simi}  and
  the discussion preceding the proof of
  Proposition~\ref{metric},
  for each $g\in G_A$, the map $T_g: \R^2\ra \R^2$
is  a  similarity    with respect to  the quasimetrics
      $D_e$,  $D_s$  and $D$.
           Hence,  in particular,
the images of the unit square  $S$   under the action of  $G_A$   on
$\R^2$ are $K$-quasi-balls in these quasimetrics for some fixed $K$.
In Lemmas  \ref{3.2} through \ref{l3.4}, $\R^2$ is equipped with one
of the three quasimetrics. 

\b{Le}\label{3.2} The coordinate function $y: \R^2\ra \R$  has
locally  finite $(2,L)$-variation
  for any $L$.
  \end{Le}

  \b{proof} Let $U\subset \R^2$ be any bounded open subset.
First observe that if two $L$-quasi-balls have comparable size, then
the oscillation of $y$ over the two quasi-balls will be comparable.
Hence when we calculate  the $2$-variation, it suffices to consider
only packings  of $U$  by quasi-balls of the form $T_g(S)$  where
$g\in G_A$. For each such square, we clearly have
$$[osc(y|_B)]^2=area(B)$$
 where $area(B)$  denotes the Euclidean area.  It follows that the
 $2$-variation of $y|_U$   is bounded by the area
 of $U$.

  \end{proof}

\b{Le}\label{l3.3}
   Let $U\subset \R^2$ be an open subset.
 If $u:  U\ra \R$   is a continuous function which
is  not constant along  some horizontal  line segment in $U$, then
$V_{2, K}(u)=\i$.

\end{Le}

\b{proof}
    Since $u$ is continuous and is not constant along a horizontal
    line segment in $U$,  after composing $u$ with an affine function,
      we may assume that  there is a rectangle $C=[a,b]\times [c,d]\subset
      U$ such that
 $u\le 0$ on $F_0:=\{(x,y)\in C: x=a\}$  and $u\ge 1$ on $F_1:=\{(x,y)\in C:  x=b\}$.
      Let
${\mathcal{G}}$  be the standard unit coordinate grid.  Pick $t\in
\R$, $t<<0$.   The image of ${\mathcal{G}}$ under
$$\lambda_t=\left[\begin{array}{cc} e^t & te^t\\
0 &  e^t  \end{array} \right]$$
 is a \lq\lq sheared grid", whose tiles
  have area $e^{2t}$.
Organize   these into nearly horizontal
  chains (which correspond  to the image of vertical strips under
  $\lambda_t$).  Notice that these chains have  slope   $1/t$ and intersect
  vertical lines in  segments with Euclidean length  $e^t/|t|$.
  It follows that there are at least
$$\frac{(d-c)-(b-a)/|t|}{e^t/|t|}=\frac{(d-c)|t|-(b-a)}{e^t}$$
  such chains  connecting the left edge $F_0$  of
  $C$ to the right edge $F_1$ of $C$.

Now consider a chain as above that connects   $F_0$ and $F_1$.
Orient the chain from
  left to right. Let $T$ be the last tile in the chain that
  intersects $F_0$ and  $T'$ the first tile in the chain that
  intersects $F_1$.  Order the tiles in the chain between $T$ and $T'$  from left to
  right and denote them by $T_1$,  $\cdots$, $T_k$.   Set $T_0=T$,
  $T_{k+1}=T'$.   Let $p_i$ ($i=1, \cdots, k+1$) be the upper left
vertex of $T_i$. Also choose any $p_0\in T_0\cap F_0$ and
$p_{k+2}\in T_{k+1}\cap F_1$. Notice that the difference between the
$x$-coordinates of $p_{i+1}$ and $p_i$ ($i=1, \cdots,
k$) is $|t|e^t$.  
   It
follows that $k<\frac{b-a}{|t|e^t}$.

Let $a_i$ be the oscillation of $u$ on $T_i\cap S$. Then $a_i\ge
|u(p_{i+1})-u(p_i)|$.    By the triangle inequality,  we have
$$\Sigma _{i=0}^{i={k+1}} a_i\ge\Sigma _{i=0}^{i={k+1}}|u(p_{i+1})-u(p_i)|
\ge |u(p_{k+2})-u(p_0)|\ge 1.$$
  In the last inequality  we used the facts that $u\le 0$ on $F_0$ and $u\ge 1$ on
  $F_1$.
  Hence
  $$\Sigma _{i=0}^{i={k+1}} a_i^2\ge \frac{1}{k+2}(\Sigma _{i=0}^{i={k+1}}
  a_i)^2\ge\frac{1}{k+2}\ge  \frac{|t|e^t}{(b-a)+2|t|e^t}.$$
   Since there are at least  $\frac{(d-c)|t|-(b-a)}{e^t}$ chains connecting
     $F_0$  and $F_1$, the $(2, K)$ -- variation of $u$ over this
   particular packing is
     at least
     $$\frac{|t|e^t}{(b-a)+2|t|e^t}\times \frac{(d-c)|t|-(b-a)}{e^t}=\frac{|t|\big\{(d-c)|t|-(b-a)\big\}}{(b-a)+2|t|e^t}.$$
 As $t\ra -\i$,  we see that  $V_{2,K}(u)=\i$.

\end{proof}

\b{Le} \label{l3.4} Let $U, V\subset \R^2$ be two open subsets, and
$F: U\ra V$ be a quasisymmetric map. Then $F$ maps each horizontal
line segment  in $U$ to a horizontal line segment in $V$.

\end{Le}

\b{proof} Assume   $F: U\ra V$ is  $\eta$-quasisymmetric. Suppose
that the claim in the lemma is false. Then there are two points $p,
q$ on the same horizontal line segment in $U$ such that $F(p)$  and
$F(q)$
  are not on the same horizontal line. Then $F(p)$ and $F(q)$ have
  different $y$ coordinates. Hence $y\circ F$ is not  constant
   along horizontal  lines. By Lemma \ref{l3.3},  $V_{2, K}(y\circ F)=\i$.
     On the other hand,  applying Lemma \ref{l3.1}  to the
function $y\circ F: U\ra \R$ and $F: U\ra V$, we obtain $V_{2,
\eta(K)}(y)\ge V_{2, K}(y\circ F)=\i$. This contradicts Lemma
\ref{3.2}.

\end{proof}

\b{Prop} \label{p3.4} Let $F: \p G_A\ra \p G_A$ be a quasisymmetric
  homeomorphism, where $\p G_A$ is equipped with a visual metric.
  Then $F$ fixes the point $\xi_0$  and maps horizontal lines to
   horizontal   lines.

\end{Prop}

\b{proof} Suppose $F(\xi_0)\not=\xi_0$. Then $F$ induces a
homeomorphism
$$F_1: \p G_A\backslash\{\xi_0, F^{-1}(\xi_0)\}\ra \p
G_A\backslash\{F(\xi_0), \xi_0\}$$
  between two open subsets of
$\R^2$.  Since a visual metric (away from $\xi_0$) is locally
quasisymmetrically equivalent with a parabolic visual metric (say a
metric of the form  $D_e^\epsilon$ with $\epsilon$ sufficiently
small) (see \cite{SX} Section 5),  $F_1$    is locally
quasisymmetric with respect to any one of $D_e$, $D_s$ and $D$.
Now Lemma \ref{l3.4}
     implies that $F_1$
maps horizontal line segments to  horizontal  line segments. Let $L$
be a complete horizontal line in $\R^2$ which does not contain
$F^{-1}(\xi_0)$.   Then $L\cup \{\xi_0\}$ is a circle in $\p G_A$
and hence $F(L\cup \{\xi_0\})$ is a circle in $\R^2$.
   By the above argument, $F(L)$ is horizontal and is dense in the
   circle   $F(L\cup \{\xi_0\})\subset \R^2$.  This is clearly impossible.
    Hence $F$ fixes $\xi_0$.   Now Lemma  \ref{l3.4}  implies $F$ maps horizonal
    lines to horizontal lines.

\end{proof}

We omit the proof of the following consequence of Proposition
\ref{p3.4} since the  proof is more or less routine and is already
contained in \cite{SX}, Section 6.

\begin{Cor}\label{finitege}
The group  $G_A$ is not quasiisometric to any finitely generated
group.
\end{Cor}

\section{Quasisymmetric maps are $D$-biLipschitz}\label{bilip}

In this section we show that every quasisymmetric map of $\p G_A$ is
biLipschitz with respect to $D$. One should contrast this with the
round sphere or the Euclidean space, where there are plenty of
non-biLipschitz quasisymmetric maps. On the other hand, $(\R^2,D)$
is not as rigid as the ideal boundary of a quarternionic hyperbolic
space or a Cayley plane (\cite{P2}) or a Fuchsian building
(\cite{BP}, \cite{X}), where each self quasisymmetry is a conformal
map.

Let $K\ge 1$ and $C>0$. A bijection $F:X_1\ra X_2$ between two
quasimetric spaces is called a $K$-\e{quasisimilarity} (with
constant $C$) if
\[
   \frac{C}{K}\, d(x,y)\le d(F(x), F(y))\le C\,K\, d(x,y)
\]
for all $x,y \in X_1$.
   When $K=1$, we say $F$ is a \e{similarity}.
It is clear that a map is a quasisimilarity if and only if it is a
biLipschitz map. The point of using the notion of quasisimilarity is
that sometimes there is control on $K$ but not on $C$.

\begin{theorem}\label{main}
Let $F:(\R^2, D)\ra (\R^2, D)$ be an $\eta$-quasisymmetry. Then $F$
is a $K$-quasisimilarity, where $K=(\eta(1)/{\eta^{-1}(1)})^6$.
\end{theorem}

   We first recall some definitions.

Let $g: (X_1, \rho_1)\ra (X_2, \rho_2)$ be a  bijection  between two
quasimetric spaces.
   Suppose $g$ satisfies the following condition:
      for any  fixed $x\in X_1$,  $\rho_1(y,x)\ra  0$
           if and only if
      $\rho_2(g(y), g(x))\ra 0$.
We define for every $x\in X_1$  and $r>0$,
\begin{align*}
   L_g(x,r)&=\sup\{\rho_2(g(x), g(x')):   \rho_1(x,x')\le r\},\\
   l_g(x,r)&=\inf\{\rho_2(g(x), g(x')):   \rho_1(x,x')\ge r\},
\end{align*}
and set
\[
   L_g(x)=\limsup_{r\ra 0}\frac{L_g(x,r)}{r}, \ \
   l_g(x)=\liminf_{r\ra 0}\frac{l_g(x,r)}{r}.
\]
  Then
\[
  L_{g^{-1}}(g(x))=\frac{1}{l_g(x)} \ \text{ and }\ l_{g^{-1}}(g(x))=\frac{1}{L_g(x)}
\]
for any $x\in X_1$. If $g$ is an $\eta$-quasisymmetry, then
$L_g(x,r)\le \eta(1)l_g(x, r)$ for all $x\in X_1$ and $r>0$. Hence
if in addition
\[
    \lim_{r\ra 0}\frac{L_g(x,r)}{r}\ \ {\text{or}} \ \ \lim_{r\ra 0}\frac{l_g(x,r)}{r}
\]
exists, then
\[
    0\le l_g(x)\le L_g(x)\le \eta(1) l_g(x)\le \infty.
\]

We notice that for every $y_1, y_2\in \R$, the Hausdorff distance
with respect to  $D$,
\begin{equation}\label{eq:1}
  HD(\R\times \{y_1\}, \R\times \{y_2\})=|y_1-  y_2|.
\end{equation}
Also, for any $p=(x_1, y_1)\in \R^2$ and any $y_2\in \R$,
\begin{equation}\label{eq:2}
   D(p, \R\times \{y_2\})=|y_1-y_2|.
\end{equation}

Let $F:(\R^2, D)\ra (\R^2, D)$ be an $\eta$-quasisymmetry.
   By Lemma \ref{l3.4}
$F$ preserves the horizontal foliation on $\R^2$. Hence it
induces a map $G: \R\ra \R$  such that for any $y\in \R$,
$F(\R\times \{y\})=\R\times \{G(y)\}$. For each $y\in \R$, let
$H(\cdot,y):\R\ra \R$ be the map such that $F(x,y)=(H(x,y),G(y))$
for all $x\in \R$.
  Notice that the restriction of $D$ to a horizontal line agrees
  with the Euclidean distance.
Because $F:(\R^2,D)\to(\R^2,D)$ is an $\eta$-quasisymmetry,
        for each fixed $y\in
\R$, the map $H(\cdot, y): (\R, |\cdot|)\ra (\R, |\cdot|)$ is also
   an $\eta$-quasisymmetry.  
  The following lemma together with equations~(\ref{eq:1})
and~(\ref{eq:2}) imply that $G:\R\ra \R$ is also an
$\eta$-quasisymmetry  with respect to the
   Euclidean metric on $\R$.

\begin{Le}\label{tyson2} \e{(\cite[Lemma 15.9]{T})}
Let $g: X_1\ra X_2$  be an $\eta$-quasisymmetry and $A,B, C\subset
X_1$.  If  $HD(A,B)\le t\, HD(A, C)$ for some $t\ge 0$, then there
is some $a\in A$ such that
\[
   HD(g(A), g(B))\le \eta(t) d(g(a), g(C)).
\]
\end{Le}

 We recall that if $g: X_1\ra X_2$ is an $\eta$-quasisymmetry, then
  $g^{-1}:  X_2\ra X_1$ is an $\eta_1$-quasisymmetry, where
$\eta_1(t)=(\eta^{-1}(t^{-1}))^{-1}$. See \cite{V}, Theorem 6.3.

Theorem \ref{main} is proved in
 Lemmas \ref{l1} through \ref{l4}.
 In these proofs,
     the quantities $l_G, L_G,
l_{G^{-1}}, L_{G^{-1}}$ and $l_{H(\cdot,y)}$, $L_{H(\cdot,y)}$,
$l_I$ and $L_I$ are all defined with respect to the Euclidean metric
on $\R$, where $I:=H(\cdot, y)^{-1}: \R\ra \R$.

\begin{Le}\label{l1}
  The following hold for all $y\in \R$, $x\in \R$:
\newline
(1)  $L_G(y, r)\le\eta(1)\, l_{H(\cdot, y)}(x, r)$ for any
$r>0$;\newline
   (2)  $\eta^{-1}(1)\, l_{H(\cdot, y)}(x)\le l_G(y)\le \eta(1)\, l_{H(\cdot, y)}(x)$; \newline
    (3)  $\eta^{-1}(1)\, L_{H(\cdot, y)}(x)\le L_G(y)\le \eta(1)\, L_{H(\cdot,
    y)}(x)$.


   \end{Le}

   \begin{proof}
  (1)  Let $y\in \R$, $x\in \R$  and  $r>0$.
 Let $y'\in \R$   with
 $|y-y'|\le r$  and     $x'\in \R$
   with $|x-x'|\ge r$.  Denote   $x''=x+(y'-y)\ln |y'-y|$.
  Then $D((x,y), (x'', y'))\le r\le D((x, y), (x', y))$.
    Since $F$ is $\eta$-quasisymmetric, we have
\b{align*} |G(y)- G(y')|   \le D(F(x'',y'), F(x,y))  & \le \eta(1)\,
D(F(x,y), F(x', y))\\
& =\eta(1)\,|H(x,y)-H(x', y)|.
\end{align*}
  Since  $y'$ and $x'$ are arbitrary,  (1)  follows.

(2) and (3).   It follows from  $l_G(y, r)\le L_G(y, r)$,
      $l_{H(\cdot, y)}(x, r)\le  L_{H(\cdot, y)}(x, r)$
       and (1)
  that  $L_G(y, r)\le\eta(1)\, L_{H(\cdot, y)}(x, r)$
     and  $l_G(y, r)\le\eta(1)\, l_{H(\cdot, y)}(x, r)$  for any $r>0$.
  Hence  $L_G(y)\le \eta(1)\, L_{H(\cdot, y)}(x)$
    and $l_G(y)\le \eta(1)\, l_{H(\cdot, y)}(x)$.
  Notice that the inverse map $F^{-1}:(\R^2,D)\ra (\R^2,D)$
is an $\eta_1$-quasisymmetry. 
Applying
  the inequality  $l_G(y)\le \eta(1)\, l_{H(\cdot, y)}(x)$
  to $I:=H(\cdot, y)^{-1}$ and $G^{-1}$ we obtain:
\[\frac{1}{L_G(y)}=l_{G^{-1}}(G(y))\le \eta_1(1)\cdot  l_I(H(x,y))=\frac{1}{\eta^{-1}(1)} \cdot
\frac{1}{L_{H(\cdot,y )}(x)},
\]
  hence   $L_G(y)\ge  \eta^{-1}(1) L_{H(\cdot, y)}(x)$.
  Similarly we prove   $l_G(y)\ge  \eta^{-1}(1) l_{H(\cdot, y)}(x)$.

   \end{proof}

Because $G: \R\ra \R$ is a
   quasisymmetry,  it  is differentiable a.e. (with respect to the
Lebesgue measure).

\begin{Le}\label{verticaldila}
  Let   $y\in \R$ be such that
  $G'(y)$ exists.  Then
    $0<l_G(y)=L_G(y)=G'(y)<\infty$.
\end{Le}

\begin{proof}  
Let  $y\in Y$ be such  that  $G'(y)$ exists.
  Then $0\le l_G(y)=L_G(y)=G'(y)<\i$.
  Suppose $G'(y)=0$.  Then  Lemma \ref{l1} (3) implies
  $L_{H(\cdot, y)}(x)=0$  for  all  $x\in \R$.  It follows that
   $H(\cdot, y):  \R\ \ra   \R$ is a constant function, contradicting the fact that
$H(\cdot, y)$   is a homeomorphism.  Hence $G'(y)\not=0$.

\end{proof}

\begin{Le}\label{l2}
 Let   $y\in \R$ be such that
  $G'(y)$ exists. Then  the map $H(\cdot, y):\R\ra\R$
is an $\eta(1)/\eta^{-1}(1)$-quasisimilarity  with constant $G'(y)$.
\end{Le}

\begin{proof}
By Lemma \ref{l1} (2) we have $l_{H(\cdot, y)}(x)\ge
l_G(y)/{\eta(1)} $  for all  $x\in \R$.   Lemma \ref{l1} (3) and
Lemma \ref{verticaldila}  imply $L_{H(\cdot,  y)}(x)\le
L_G(y)/\eta^{-1}(1)=l_G(y)/\eta^{-1}(1)$  for all $x\in  \R$.
  Because $\R$ is a geodesic space,   the map
$H(\cdot, y)$  is an $\eta(1)/\eta^{-1}(1)$-quasisimilarity  with
constant $l_G(y)=G'(y)$.   

\end{proof}

\begin{Le}\label{l3}
There exists a constant $C>0$  with the following properties:
\begin{enumerate}
\item[(1)]   For each $y\in \R$,  $H(\cdot,y)$  is an
$(\eta(1)/\eta^{-1}(1))^4$-quasisimilarity with constant $C$;
\item[(2)]  $G:\R\ra \R$ is an $(\eta(1)/\eta^{-1}(1))^5$-quasisimilarity
with constant $C$.
\end{enumerate}
\end{Le}

\begin{proof}
 (1)   Fix any  $y_0\in \R$   such that $G'(y_0)$ exists
        and set $C=G'(y_0)$. Let $y\in \R$
          be any point such that
            $G'(y)$ exists.  By
Lemma \ref{l2},  the map $H(\cdot, y):\R\ra\R$ is an
$\eta(1)/\eta^{-1}(1)$-quasisimilarity  with constant $G'(y)$.  
   Let  $x_0\in \R$ and
choose $x\in \R$ such that $|x-x_0|\ge |y-y_0|$. Let
$x'=x+(y_0-y)\ln|y_0-y|$.  
Then
\[
   D((x',y_0),(x_0,y))=D((x,y),(x_0,y))=|x-x_0|.
\]
By picking $x$ so that in addition
$$\kappa:=\big\vert H(x',y_0)-H(x_0,y)-(G(y_0)-G(y))\ln |G(y_0)-G(y)|\big\vert >|G(y_0)-
G(y)|,$$
 by the $\eta$-quasisymmetry of $F$ we have
\begin{align*}
  \kappa =D(F(x', y_0), F(x_0, y))
    & \le \eta(1) D(F(x,y), F(x_0,y)) =\eta(1) |H(x, y)- H(x_0, y)|.
\end{align*}
By  Lemma \ref{l2}  and the choice of $y$, we have
   $$|H(x,y)-
H(x_0,y)|\le(\eta(1)/\eta^{-1}(1)) l_G(y)|x- x_0|.$$
    On the other
hand, letting  $\tau=(G(y_0)-G(y))\ln |G(y_0)-G(y)|$,  we have
\begin{align*}
 \kappa& \ge |H(x',y_0)- H(x_0,y_0)|- |H(x_0,y_0)- H(x_0,y)|-|\tau|\\
               & \ge\frac{G'(y_0)}{\eta(1)/\eta^{-1}(1)} |x'-x_0|-|H(x_0,y_0)- H(x_0,y)|-|\tau|.
\end{align*}
      Combining the above inequalities and letting $|x-x_0|\ra \infty$, we obtain
\[
G'(y)=l_G(y)\ge \frac{1}{(\eta(1))^3 (\eta^{-1}(1))^{-2}} G'(y_0)
   =\frac{C}{(\eta(1))^3 (\eta^{-1}(1))^{-2}}.
\]
 Switching the roles of $y$ and $y_0$ we obtain
\[
G'(y_0)\ge \frac{1}{(\eta(1))^3 (\eta^{-1}(1))^{-2}} G'(y).
\]
  Hence  $ \frac{C}{(\eta(1))^3 (\eta^{-1}(1))^{-2}}\le  G'(y)\le C(\eta(1))^3 (\eta^{-1}(1))^{-2}$.
 By  Lemma~\ref{l1},  for all $x\in \R$,
\[
  L_{H(\cdot, y)}(x)\le  L_G(y)/\eta^{-1}(1)\le C (\eta(1))^3
  (\eta^{-1}(1))^{-3}
\]
and
\[
   l_{H(\cdot, y)}(x)\ge \frac{1}{\eta(1)}l_G(y)\ge \frac{C}{(\eta(1))^4(\eta^{-1}(1))^{-2}}.
\]
  Hence  for   a.e. $y\in \R$,  the map
$H(\cdot,y)$  is an $(\eta(1)/\eta^{-1}(1))^4$-quasisimilarity with
constant $C$.   A limiting argument shows that this is true for all
$y$.

(2)  Statement  (1)   implies the following  for all $x, y\in \R$,
$$\frac{C}{(\eta(1)/\eta^{-1}(1))^4}\le  l_{H(\cdot, y)}(x)\le  L_{H(\cdot, y)}(x)\le C(\eta(1)/\eta^{-1}(1))^4.$$
  Now Lemma \ref{l1} implies
$$\frac{C}{(\eta(1)/\eta^{-1}(1))^5}\le  l_G(y)\le  L_G(y)\le C(\eta(1)/\eta^{-1}(1))^5$$
  for all $y\in \R$.  Hence (2) holds.

\end{proof}

\begin{Le}\label{l4}
   $F$ is an $(\eta(1)/\eta^{-1}(1))^6$-quasisimilarity with
constant $C$,    where $C$ is the constant in Lemma~\ref{l3}.
\end{Le}

\begin{proof}  Set $K=(\eta(1)/\eta^{-1}(1))^5$.
Let $(x_1,y_1),(x_2,y_2)\in\R^2$. We shall first establish a lower
bound for    $D(F(x_1,y_1),F(x_2,y_2))$.
    Set $\tau=x_1- x_2-(y_1-y_2)\ln |y_1-y_2|$.
If $|\tau|\le |y_1-y_2|$,
then $D((x_1, y_1),(x_2,y_2))=|y_1- y_2|$ and by
  Lemma \ref{l3} (2),
\begin{align*}
 D(F(x_1,y_1), F(x_2, y_2))\ge |G(y_1)-G(y_2)|
     \ge \frac{C}{K} |y_1-y_2|
     =\frac{C}{K} D((x_1,y_1),(x_2,y_2)).
\end{align*}
 If $|\tau|>|y_1-y_2|$, then
\[
  D((x_1,y_1),(x_2,y_2))=|\tau|=D((x_1-(y_1-y_2)\ln
|y_1-y_2|,y_2),(x_2,y_2)),
\]
and since $F$ is an $\eta$-quasisymmetry, we have
\begin{align*}
 D(F(x_1,y_1),F(x_2,y_2)) &\ge\frac{1}{\eta(1)} D(F(x_1-(y_1-y_2)\ln
|y_1-y_2|, y_2),F(x_2, y_2))\\
         &=\frac{1}{\eta(1)} \big\vert H(x_1-(y_1-y_2)\ln
|y_1-y_2|,y_2)- H(x_2,y_2)\big\vert \\
         &\ge \frac{C}{\eta(1)K}\big \vert  x_1- x_2-(y_1-y_2)\ln
|y_1-y_2|\big\vert \\
         &=\frac{C}{\eta(1)K}D((x_1,y_1),(x_2,y_2)),
\end{align*}
with the second inequality following from  Lemma \ref{l3} (1). Hence
we have a lower bound for $D(F(x_1,y_1),F(x_2,y_2))$.

By Lemma \ref{l3} (2), $G^{-1}:\R\ra \R$ is a $K$-quasisimilarity
with constant $C^{-1}$. Similarly, Lemma \ref{l3} (1) implies that
for each $y\in \R$, $(H(\cdot,y))^{-1}$ is a $K$-quasisimilarity
with constant $C^{-1}$. Also recall that $F^{-1}$ is an
$\eta_1$-quasisymmetry and $F$ is an $\eta$-quasisymmetry. Now the
argument in the previous paragraph applied to $F^{-1}$ implies
\[
   D(F^{-1}(x_1, y_1), F^{-1}(x_2, y_2))\ge
         \frac{1}{CK\eta_1(1)}D((x_1,y_1),(x_2,y_2)).
\]
It follows that
\[
  D(F(x_1,y_1),F(x_2,y_2))\le CK\eta_1(1) D((x_1, y_1),(x_2,y_2))=CK/\eta^{-1}(1) D((x_1, y_1),(x_2,y_2))
\]
for  all $(x_1,y_1),(x_2,y_2)\in \R^2$. Hence we also obtain  an
upper bound for the quantity $D(F(x_1, y_1),F(x_2,y_2))$.
\end{proof}

\section{Characterization of quasisymmetric maps}\label{chara}

In this section we give a complete description of all self
quasisymmetric maps of $\p G_A$.

\b{theorem}\label{Main}
  {A map $F: (\R^2, D)\ra (\R^2, D)$ is a quasisymmetric map
  if and only if it has the following form: $F(x,y)=(ax+c(y), ay+b)$
  for all $(x,y)\in \R^2$, where $a\not=0$,  $b$ are constants and
  $c: \R\ra \R$ is a Lipschitz map.}

  \end{theorem}

  Let $F:(\R^2, D)\ra (\R^2, D)$ be  a quasisymmetric map.
   From Section \ref{bilip}, we know  there is a  quasisymmetric map $G: \R\ra \R$, and for each
   $y\in \R$ there is a quasisymmetric map $H(\cdot,y): \R\ra \R$
   such that  $F(x,y)=(H(x,y), G(y))$ for all $(x,y)\in \R^2$.
     Then $G'(y)$ exists almost everywhere.  Similarly, for each $y\in \R$,
      the map $H(\cdot, y)$ has derivative   $H_x(x,y)$ for a.e. $x\in \R$.

  \b{Le}\label{c.l1}
  {Let $F:(\R^2, D)\ra (\R^2, D)$ be  a quasisymmetric map.
  Let $y\in \R$ be such that $G'(y)$ exists,  and $x\in \R$ such
  that  $H_x(x, y)$ exists at $x$.  Then
    $G'(y)=H_x(x,y)$.}

    \end{Le}

    \b{proof}
      Let $F:(\R^2, D)\ra (\R^2, D)$ be  an  $\eta$-quasisymmetric map.
By replacing $F$ with $T_{(-H(x,y), -G(y))}\circ F\circ T_{(x,y)}$,
we may assume $(x,y)=(H(x,y), G(y))=0$.  Here $T_{(x,y)}$ denotes
the  Euclidean translation by $(x,y)$.
  Lemma \ref{verticaldila} implies $G'(0)\not=0$.
By composing $F$ with a dilation $\lambda_t$ for a suitable $t$ we
may assume $G'(0)=1$ or $-1$.   If $G'(0)=-1$, we  further compose
$F$ with the rotation
 $\pi: \R^2\ra \R^2$, $\pi(x,y)=(-x,-y)$.  Hence we may assume
 $G'(0)=1$.
   Denote $\lambda=H_x(0,0)$.
   By Lemma \ref{l3} (2)  we have $\lambda\not=0$.
      We shall prove that $\lambda=1$.

Since  $\lambda_t$ is a similarity, the family of maps
$\{F^t:=\lambda_t\circ F\circ \lambda_{-t}| t\in \R\}$ consists of
$\eta$-quasisymmetric maps. Write $F^t(x,y)=(H^t(x,y),G^t(y))$.
  We notice that $H^t(x, 0)=e^t H(e^{-t} x, 0)$ and $G^t(y)=e^t
  G(e^{-t} y)$.  Since the derivative
   $H_x(0,0)$ exists,
   the maps $H^t(\cdot, 0): \R\ra \R$
  converge (as $t\ra \i$)
  in the pointed Gromov-Hausdorff distance towards the map
   $x \ra \lambda x$.   Similarly,  the maps $G^t: \R\ra \R$
   converge  (as $t\ra \i$)
   in the pointed Gromov-Hausdorff distance towards the
   map $y\ra y$.  The compactness property of quasisymmetric maps
     implies    that
     there is a sequences $t_i\ra \i$ such that
     $F^{t_i}$ converges in the pointed Gromov-Hausdorff distance
     towards an $\eta$-quasisymmetric map $\tilde F: (\R^2, D)\ra
     (\R^2, D)$.  If we write $\tilde F(x,y)=(\tilde H(x,y), \tilde
     G(y))$, then $\tilde G(y)=y$  and $\tilde H(x, 0)=\lambda x$.

By   Theorem \ref{main},   the map $\tilde F$  is $L$-biLipschitz
for some $L\ge 1$.
 Fix some $x\in \R$ and a positive integer $n\ge
1$.
  For $i=0, \cdots,  n$, let $(x_i, y_i)=(x-\frac{i}{n}\ln n,
  \frac{i}{n})$.   Then  $D((x_i, y_i), (x_{i+1}, y_{i+1}))=1/n$.
Hence
  $$|\tilde H(x_i, y_i)-\tilde H(x_{i+1}, y_{i+1})-\frac{1}{n}\ln
  n|\le D(\tilde F(x_i, y_i), \tilde F(x_{i+1},y_{i+1}))\le L\cdot
  \frac{1}{n}.$$ Adding  up all these inequalities for $i=0, \cdots,n-1$
    and using the  triangle inequality
  we obtain
\begin{equation}\label{e1}
|\tilde H(x_0, y_0)-\tilde H(x_n, y_n)-\ln n|\le L.
\end{equation}
      On the other hand, $D((x_n, y_n), (x-\ln n, 0))=1$  and hence
\b{equation}\label{e2}
      |\tilde H(x_n, y_n)- \tilde H(x-\ln
n,0)|\le D(\tilde F(x_n,y_n), \tilde F(x-\ln n, 0))\le L.
\end{equation}
 It follows from (\ref{e1})  and (\ref{e2})   that
 $|\tilde H(x_0, y_0)- \tilde H(x-\ln n,0)-\ln n|\le 2L$.
   Notice that $\tilde H(x_0, y_0)=\tilde H(x, 0)=\lambda x$  and
   $\tilde H(x-\ln n,0)=\lambda (x-\ln n)$. So we have
     $|(\lambda-1)\ln n|\le 2L$. Since this is true for all $n\ge
     1$, we must have $\lambda=1$.

    \end{proof}

\b{Le}\label{cl.2}
  {There exist    constants $a\not=0$  and   $b$
    and also a function $c:\R\ra \R$  such that\newline
  (1) $G(y)=ay+b$;\newline
  (2) $H(x,y)=ax+c(y)$ for all $(x,y)\in \R^2$.}

  \end{Le}

  \b{proof}
Let $y\in \R$ be any point where $G$ is differentiable.  By Lemma
\ref{c.l1}, the quasisymmetric map $H(\cdot, y): \R\ra \R$ a.e. has
derivative $G'(y)$.  It follows that $H(\cdot, y)$ is an affine map;
 to be more precise, there is a constant $c(y)$ depending only on $y$ such that
  $H(x,y)=G'(y) x+c(y)$ for all $x\in \R$.

  We claim that $G'(y_1)=G'(y_2)$ holds for any two points
  $y_1,y_2\in \R$ at which $G$ is differentiable.
   Set  $\tau=(y_2-y_1)\ln |y_2-y_1|$.
    Let $x>0$  and denote
     $p=(0, y_1)$, $q=(x, y_1)$,
      $p'=(\tau, y_2)$  and $q'=(x+\tau, y_2)$.
  One checks that $D(p,q)=D(p',q')=x$ and $D(p, p')=D(q,
  q')=|y_2-y_1|$.
    By Theorem \ref{main} $F$ is $L$-biLipschitz for some
     $L\ge 1$.     We have $D(F(p), F(p'))\le L|y_2-y_1|$ and $D(F(q),
    F(q'))\le L |y_2-y_1|$.  On the other hand, by the preceding
    paragraph, we have   $F(p)=(c(y_1),  G(y_1))$,
    $F(q)=(G'(y_1)x+c(y_1),  G(y_1))$  and
    $F(p')=(G'(y_2)\tau+c(y_2), G(y_2))$,
      $F(q')=(G'(y_2)(x+\tau)+c(y_2),  G(y_2))$.
        Set  $\tau'=(G(y_2)-G(y_1))\ln |G(y_2)-G(y_1)|$.
        Since
        $$|[G'(y_2)(x+\tau)+c(y_2)]-[G'(y_1)x+
    c(y_1)]-  \tau'|\le
        D(F(q), F(q'))\le L |y_2-y_1|$$  for all $x>0$, we must have
  $G'(y_1)=G'(y_2)$.

 Since $G$ is differentiable a.e.,
 it follows from the above claim that $G$ a.e.has constant
 derivative,  hence must be an affine map. That is, there are
 constants $a\not=0$, $b$ such that $G(y)=ay+b$ for all $y\in \R$.
 This proves (1).  Now (2) follows from (1) and the first paragraph.

  \end{proof}

\begin{proof}[\bf{Completing the proof of Theorem~\ref{Main}.}]
First suppose $F: (\R^2, D)\ra (\R^2, D)$ is a quasisymmetric map.
Then by Lemma \ref{cl.2}
  $F$ has the form $F(x,y)=(a x+ c(y), ay+b)$, where
 $a\not=0$,  $b$   are constants,  and $c: \R\ra \R$ is a function.
Now fix $y_1,y_2\in \R$. Let $\tau=(y_2-y_1)\ln |y_2-y_1|$ and
denote $p=(0, y_1)$, $q=(\tau, y_2)$.  One checks that
$D(p,q)=|y_2-y_1|$.
  By Theorem \ref{main}  $F$ is a $L$-biLipschitz map for some  $L\ge 1$.
    Hence
    $D(F(p), F(q))\le L D(p, q)=L|y_2-y_1|$.
     On the other hand,  $F(p)=(c(y_1), ay_1+b)$ and
     $F(q)=(a\tau+c(y_2),  ay_2+b)$. We have
\b{align*} D(F(p), F(q)) & \ge
     \big\vert (a\tau+c(y_2)-c(y_1))-a(y_2-y_1)\ln |a(y_2-y_1)|\big\vert\\
    & =\big\vert c(y_2)-c(y_1)-(a \ln |a|) (y_2-y_1)\big\vert.
     \end{align*}
     Now the triangle
     inequality implies $|c(y_2)-c(y_1)|\le \big(L+\big\vert a\ln
     |a|\big\vert\big)|y_2-y_1|$, that is, $c$ is $\big(L+\big\vert a\ln
     |a|\big\vert\big)$-Lipschitz.

     Conversely,    suppose $F$ has the form
$F(x,y)=(a x+ c(y), ay+b)$, where
 $a\not=0$,  $b$   are constants,  and $c: \R\ra \R$ is  $L$-Lipschitz.
One checks by direct calculation that $F$ is Lipschitz, as follows.
  Let $p=(x,y), q=(x', y')\in \R^2$  be two arbitrary points.  Then
  $F(p)=(ax+c(y), ay+b)$ and $F(q)=(ax'+c(y'), ay'+b)$.
  Set  $\tau=(x'-x)-(y'-y)\ln |y'-y|$.
  We have
   $D(p, q)=\max\{|y'-y|, |\tau|\}$
     and
     $$D(F(p), F(q))=\max\big\{\big\vert a(y'-y)\big\vert,  \big\vert a\tau+[c(y')-c(y)]-(a\ln |a|)(y'-y)\big\vert\big\}.$$
       Now  $|a(y'-y)|=|a|\cdot |y'-y|\le |a|D(p,q)$
         and
\b{align*} \big\vert a\tau+[c(y')-c(y)]-(a\ln |a|)(y'-y)\big\vert
&\le \big\vert a\tau\big\vert +
\big\vert c(y')-c(y)\big\vert +\big\vert (a\ln |a|)(y'-y)\big\vert\\
 &   \le |a|
     D(p,q)+L|y'-y|+\big\vert a\ln |a|\big\vert\cdot |y'-y|\\
     & \le \big(|a|+L+\big\vert a\ln |a|\big\vert\big) D(p,
     q).
     \end{align*}
       It follows that $F$ is Lipschitz with Lipschitz constant
        $|a|+L+\big\vert a\ln |a|\big\vert.$
            On the other hand, $F^{-1}$ has the form
            $$F^{-1}(x, y)=\left(\frac{1}{a}\cdot x -
            \frac{1}{a}\cdot
            c\left(\frac{1}{a}y-\frac{b}{a}\right), \; \, \frac{1}{a}\cdot y-\frac{b}{a}\right).$$
              As a composition of Lipschitz maps,  the map $c':\R\ra \R$,
              $c'(y)=-\frac{1}{a}\,c(\frac{1}{a}y-\frac{b}{a})$ is also
              Lipschitz.  Hence the above calculation shows that
              $F^{-1}$ is also Lipschitz.

\end{proof}

\section{A Liouville type theorem for $(\R^2, D)$}\label{s6}

In this section we prove a
  Liouville type theorem for $(\R^2, D)$,   which says  that
     all conformal
  maps of $(\R^2, D)$ are boundary maps of isometries of
     $G_A$. We first  identify all the conformal maps of $(\R^2,
     D)$.

Using Theorem \ref{Main}, we can identify all the isometries and
similarities of $(\R^2, D)$.
 Recall that  the map $\pi$ and  similarities
$\lambda_t$ are defined in Proposition \ref{iso.and.simi}.

\b{Prop}\label{iso.and} {(1) The group of all isometries
      of $(\R^2,
D)$ is generated by Euclidean translations and $\pi$;\newline
 (2) The group of all  similarities of $(\R^2, D)$ is generated by
  Euclidean translations, $\pi$ and the similarities $\lambda_t$
  ($t\in \R$).

}\end{Prop}

\b{proof} We only prove (2), the proof of (1) being similar. Let $F:
(\R^2, D)\ra (\R^2, D)$ be a similarity. By composing   $F$  with
 a suitable $\lambda_t$, we  may assume $F$ is an isometry.
  By Theorem \ref{Main}, $F$ has the form
   $F(x,y)=(ax+c(y), ay+b)$,   where $a\not=0$,  $b$ are constants and
  $c: \R\ra \R$ is a Lipschitz map.
  By considering the restriction of $F$ on a horizontal
    line $\R\times \{y\}$, we see $a=1$ or $-1$.  By composing with
    $\pi$ if necessary (when $a=-1$), we may assume $a=1$.  By
    further composing $F$ with an   Euclidean translation, we may
    assume $b=0$ and $c(0)=0$.  Now $F$ has the form
    $F(x,y)=(x+c(y), y)$ for all $(x, y)\in \R^2$, where $c(0)=0$.
    We claim $c(y)=0$ for all $y\in \R$. Suppose $c(y)\not=0$ for
    some $y\not=0$.  Let $\epsilon=1$ or $-1$ be such that
     $\epsilon y$ and $c(y)$ are either both positive or both
     negative.   Let $p=(0,0)$ and  $q=(\epsilon y+y \ln |y|, y)$.
       Then $F(p)=p$ and $F(q)=(\epsilon y+y \ln |y|+c(y), y)$.
         One
       calculates   $D(F(p), F(q))=|\epsilon
       y+c(y)|>|y|=D(p,q)$, contradicting
          the fact that $F$ is an
       isometry. Hence $c(y)=0$ for all $y$ and $F$ is the identity
       map.

\end{proof}

Let  $X$ and $Y$  be quasimetric spaces with finite Hausdorff
dimension. Denote by $H_X$ and $H_Y$  their Hausdorff dimensions and
by $\mathcal{H}_X$  and $\mathcal{H}_Y$  their Hausdorff measures
(see
  \cite{F}   for definitions).  We say  a quasisymmetric map $f: X\ra Y$ is
  conformal if:\newline
  (1) $L_f(x)=l_f(x)\in (0, \i)$  for $\mathcal{H}_X$-almost every
  $x\in X$;\newline
  (2) $L_{f^{-1}}(y)=l_{f^{-1}}(y)\in (0, \i)$ for
  $\mathcal{H}_Y$-almost every $y\in Y$.

\b{Le}\label{liu}
  {Every conformal map $F: (\R^2, D)\ra (\R^2, D)$  is a
  similarity.
  }

  \end{Le}

  \b{proof}
Since $F$ is conformal, it  is quasisymmetric in particular.  By
Theorem \ref{intromain}, $F$ has the following form:  $F(x,y)=(ax+
c(y), ay+b)$, where $a\not=0$,  $b$ are constants and
  $c: \R\ra \R$ is a Lipschitz map.
    By composing $F$ with a similarity, we may assume
    $a=1$  and  $b=0$; that is, $F$ has the form
      $F(x,y)=(x+ c(y), y)$.  We shall prove that $c(y)$ is a
      constant function. 

      Since $c:  \R\ra \R$   is a Lipschitz function, it is differentiable
      a.e.  We shall show that $c'(y)=0$ for a.e. $y\in \R$.
         By the definition of a conformal map,
           $L_F(x,y)=l_F(x,y)$ for a.e. $(x,y)\in \R^2$ with respect
           to the  Lebesgue   measure in $\R^2$.   It follows from
           Fubini's theorem that for  a.e. $y\in \R$,  the
           derivative $c'(y)$ exists and
            $L_F(x,y)=l_F(x,y)$ for a.e. $x\in \R$.
              Let $y_0$ be   an arbitrary  such  point  and $x_0\in \R$ be such that
              $L_F(x_0,y_0)=l_F(x_0,y_0)$. We will show
              $c'(y_0)=0$.

By pre-composing and post-composing with Euclidean translations if
necessary,
  we may assume that  $(x_0, y_0)=(0, 0)$ and $c(y_0)=0$. We need to show
  $c'(0)=0$.
We will suppose $c'(0)\not=0$ and get a contradiction.
   Notice that $F(x,0)=(x,0)$ for all  $x\in \R$.
   It follows that $L_F(0,0)\ge 1$ and $l_F(0,0)\le 1$.
     Combining this with the assumption
$L_F(0,0)=l_F(0,0)$, we obtain  $L_F(0,0)=l_F(0,0)=1$. First suppose
$c'(0)>0$.  Then $c(y)>0$ for sufficiently small $y>0$.
 Let $p=(0,0)$ and $q=(r+r\ln r, r)$ with $r>0$.
  Then $F(p)=p$ and $F(q)=(r+r\ln r+c(r), r)$.
   One calculates  $D(p,q)=r$
    and $D(F(p), F(q))=r+c(r)$.
 It follows that $L_F(p,r)\ge r+c(r)$ and hence
  $L_F(p)\ge 1+c'(0)>1$,  contradicting  $L_F(0,0)=1$.
If $c'(0)<0$, then letting $q=(-r+r\ln r, r)$ one similarly obtains
a contradiction.

  \end{proof}

  \b{theorem}\label{liouville}
  {Let  $F: (\R^2,D)\ra (\R^2, D)$  be  a conformal map.  Then
    $F$ is the boundary map of some
     isometry $f:  G_A\ra G_A$.

  }

  \end{theorem}

  \b{proof}
By Lemma \ref{liu},  $F$ is a similarity. By  Proposition
\ref{iso.and} (2),  $F$ is the composition of Euclidean
translations, $\tau$ and similarities $\lambda_t$.
  Now the theorem follows from the following facts (see the end of Section \ref{metriconb}):  (1)  Euclidean translations
  of $\R^2$ are boundary maps of the Lie group left translations $L_g$
  for elements of the form $g=((x,y),0)\in G_A$;  (2) $\tau$ is the
  boundary map of the isometry $\tau': G_A\ra G_A$;
    (3) $\lambda_t$ is the boundary map of the Lie group left translation
     $L_g$ for $g=((0,0), t)$.

  \end{proof}

\section{Quasiisometries of $G_A$}\label{s7}

In this section we calculate the  quasiisometry group of $G_A$ and
 identify all the quasiisometries of $G_A$  up to
bounded distance. From this it is easy to see that all
quasiisometries of $G_A$ are almost isometries and are
height-respecting.

We first discuss the  structure of the group $QS(\R^2,D)$ of all
quasisymmetric maps of $(\R^2,D)$.
 We identify three subgroups of   $QS(\R^2,D)$.
  Let
  $H_1=\{\lambda_t: t\in \R\}\cong \R$. 
  Let $H_2=<\tau>\cong Z_2=\{\bar 0, \bar 1\}$ be the order 2 cyclic group generated by $\tau$.
  Let  $H_3$ be the group of homeomorphisms of $\R^2$ of the form
   $F_{C,b}(x,y)=(x+C(y), y+b)$, where $b\in \R$ and $C:\R\ra \R$  is a Lipschitz
   function.
Direct calculations show that
  $H_1$ and $H_2$ commute,
 both $H_1$
and $H_2$ normalize $H_3$,
  and $H_3\cap <H_1, H_2>$ is trivial.
  On the other hand,  Theorem \ref{Main}  implies that  $QS(\R^2, D)$ is generated by $H_1$, $H_2$
  and $H_3$.
  It follows that we have the following isomorphism:
  $$QS(\R^2, D)\cong H_3 \rtimes (H_1\oplus H_2).$$

Let  $L$  be the additive group  consisting of
     Lipschitz functions $C:\R\ra \R$.
       Let $\R$ act on $L$ by
         $b* C=C\circ T_b$, for $b\in \R$ and  $C\in L$,
         where $T_b$ is the translation on $\R$ by
         $b$.  Then it is easy to check that the map
given  by $F_{C,b}\mapsto (C, b)$ defines an isomorphism from
         the group $H_3$
 to the opposite group
   $\overline{L\rtimes \R}$
         of $L\rtimes \R$. 
It now follows that we have the following isomorphism:
  $$QS(\R^2, D)\cong (\overline{L\rtimes \R}) \rtimes (\R\times Z_2).$$
  Here  the action of $\R\times \{\bar 0\}$ on $\overline{L\rtimes \R}$
  is given by
  $(t,\bar 0)*(C,b)=(C', b')$ for $(t,\bar 0)\in \R\times \{\bar 0\}$ and $(C,
  b)\in \overline{L\rtimes \R}$, where
  $$C'(y)=e^t\cdot C(e^{-t} y)+bte^t\;\;
  \text{and}\;\;  b'=e^t b;$$
    and the action of $\{0\}\times Z_2$ on
    $\overline{L\rtimes \R}$
  is given by  $(0, \bar 1)*(C, b)=(C'', -b)$, where
  $C''(y)=-C(-y)$.


 Two quasiisometries $f,g: X\ra Y$ between two metric spaces are
 said to be
 equivalent  if   $\sup\{d(f(x), g(x)):  x\in X\}<\i$.
For any metric space $X$,
 the  quasiisometry group $QI(X)$   consists of equivalence classes
 of quasiisometries $X\ra X$  and has group operation given by
 composition.

For each quasiisometry $f: G_A\ra G_A$, let $\p f: \p G_A\ra \p G_A$
be its boundary map.

   \b{theorem}\label{qigroup}
   We have $QI(G_A)\cong QS(\R^2, D)\cong
(\overline{L\rtimes \R}) \rtimes (\R\times Z_2)$,   where $L$ and
the various actions   are as described above.

   \end{theorem}

\b{proof}
  By Proposition \ref{p3.4}, each quasiisometry
   $f: G_A\ra G_A$ induces a quasisymmetric map
   $\p f: (\R^2, D)\ra (\R^2, D)$.
 Notice that an quasiisometry  $f$  of $G_A$ is at finite
distance from the identity map of $G_A$ if and only if the boundary
map
 of $f$  is the identity map  on $(\R^2, D)$.
   Hence the map $\p_1: QI(G_A)\ra QS(\R^2,D)$, $\p_1([f])=\p f$ is
well-defined and is injective. Here $[f]$ denotes the equivalence of
$f$.  On the other hand,  by \cite{BS} each element in $QS(\R^2,D)$
is the boundary map of an quasiisometry.  Hence $\p_1: QI(G_A)\ra
QS(\R^2,d)$   is also surjective.

\end{proof}

We now  identify a group  of quasiisometries of $G_A$ that is
isomorphic to $QI(G_A)$.  Let $H'_1=\{L_g: g=((0,0), t) \in G_A\}$.
 Let $H'_2=<\tau'>$, where $\tau'((x,y), t)=((-x,-y), t)$ is the
  automorphism of $G_A$ defined in    Section \ref{metriconb}.
     The two groups $H_1$ and $H_2$ consist of isometries of $G_A$.
    Let $H'_3$ be the group of homeomorphisms of $G_A$ of the
    following form:
    $$f_{C,b}: G_A\ra G_A,  \;\;\;  f_{C,b}((x,y), t)=((x+C(y), y+b), t),$$
   where  $b\in \R$ and  $C\in L$.  It is clear that $H'_3$ is isomorphic to
   $H_3\cong \overline{L\rtimes \R}$.
Using    Lemma  6.3  in \cite{SX}  and the fact that $F_{C,b}:
(\R^2, D_e)\ra (\R^2, D_e)$ is biLipschitz, it is easy to check that
each $f_{C,b}$ ($b\in \R$, $C\in L$) is an almost isometry of $G_A$.
In particular,
 $H_3$ consists of quasiisometries of $G_A$.

  Let $QI'(G_A)$ be the group of homeomorphisms
   of  $G_A$
  generated by $H'_1$,
  $H'_2$ and $H'_3$.
A similar discussion as above shows that
 $$QI'(G_A)\cong H'_3 \rtimes (H'_1\times  H'_2)\cong
\overline{L\rtimes \R}  \rtimes (\R\times Z_2),$$
  where the various actions are as described above.
 Let $\p: QI'(G_A)\ra QS(\R^2, D)$ be the map that assigns  to  each
 $f\in QI'(G_A)$ its boundary map.
  It is easy to see that $\p $
 maps $H'_i$ ($i=1,2,3$)   isomorphically onto $H_i$.  It follows
 that $\p $ is an isomorphism.  Let
$p:  QI'(G_A)\ra QI(G_A)$ be the group homomorphism that assigns  to
  each $f\in QI'(G_A)$ its equivalence class.
  Since $\p=\p_1\circ p$ (where $\p_1$ is defined in the proof of Theorem \ref{qigroup}),
   it follows from
  Theorem \ref{qigroup} that $p$  is an isomorphism.
  We obtain:

  \b{theorem}\label{qasiisog}
  {Every quasiisometry $f: G_A\ra G_A$ is at a finite distance from
    exactly  one  element of $QI'(G_A)$.

  }

  \end{theorem}

Now we can provide a proof of Corollary
   \ref{c0}.

\begin{proof}[\bf{Proof of Corollary
   \ref{c0}.}]
    Since each element in $H'_1$ and $H'_2$ is an isometry of $G_A$,
   and every  element of $H'_3$ is an almost isometry,
  we see that every element of $QI'(G_A)$ is an almost isometry.
    Now
   Corollary
   \ref{c0}
   follows from this and Theorem \ref{qasiisog}.

\end{proof}

Under    the identification of $G_A$ with $\R^2\times \R$,  we view
the map $h:\R^2\times \R\ra \R$, $h((x,y),t)=t$ as the height
function.
A quasiisometry $f:G_A\ra G_A$  is height-respecting if 
$|h(f((x,y),t))-t|$ is bounded independent of $((x,y), t)\in G_A$.
Since every element of $H'_1$, $H'_2$ and $H'_3$ is
height-respecting,  we have

\begin{Cor}\label{c2}
All self quasiisometries of $G_A$ are height-respecting.
\end{Cor}

 \addcontentsline{toc}{subsection}{References}

\noindent Address:

\noindent Xiangdong Xie: Dept. of Mathematical Sciences, Georgia
Southern University, Statesboro, GA 30460, U.S.A.\hskip .4cm E-mail:
xxie@georgiasouthern.edu

\end{document}